\documentclass[12pt]{amsart}
\usepackage{float}
\usepackage{amsmath}
\usepackage{amssymb}
\usepackage{esvect}
\usepackage{verbatim}
\usepackage{bm}
\usepackage{graphicx}
\usepackage{psfrag}
\usepackage{color}
\usepackage[dvipsnames]{xcolor}
\usepackage{yfonts}
\usepackage{hyperref}
\hypersetup{colorlinks=true, linkcolor=blue, citecolor=magenta, urlcolor=wine}
\usepackage{url}
\usepackage{algpseudocode}
\usepackage{fancyhdr}
\usepackage{mathtools}
\usepackage{tikz-cd}
\usepackage{stmaryrd}
\usepackage{calrsfs}
\usepackage{quiver}

\usepackage{adjustbox}

\newtheorem*{maintheorem*}{Main Theorem}
\newtheorem{theorem}{Theorem}[section]
\newtheorem*{theorem*}{Main Theorem}

\newtheorem{corollary}[theorem]{Corollary}

\newtheorem{lemma}[theorem]{Lemma}

\theoremstyle{definition}
\newtheorem{definition}[theorem]{Definition}

\numberwithin{equation}{section}

\newcommand{\fa}{\mathfrak{a}}
\newcommand{\bb}{\mathfrak{b}}
\newcommand{\dd}{\mathfrak{d}}

\newcommand{\Gal}{\text{Gal}}
\newcommand{\ra}{\rightarrow}
\newcommand{\res}{\text{res}}
\newcommand{\C}{\mathbb{C}}

\newcommand{\p}{\mathfrak{p}}
\newcommand{\QQ}{\mathfrak{Q}}
\newcommand{\q}{\mathfrak{q}}
\newcommand{\Q}{\mathbb{Q}}

\newcommand{\Z}{\mathbb{Z}}
\newcommand{\N}{\mathbb{N}}
\newcommand{\R}{\mathbb{R}}

\newcommand{\Frob}{\operatorname{Frob}}

\colorlet{darkgreen}{green!35!black}

\begin{document}
	
	\mbox{}
	\title{Connected components of the ranges of twisted divisor functions on number fields}
	\author{Sophie Zhu}
	\address{Harvard College}
	\email{sophiezhu@college.harvard.edu}
\date{}

\begin{abstract} 
    Let $r\in\mathbb{C}$, let $K$ be a finite extension of $\mathbb{Q}$, let $I_K$ be the monoid of integral ideals in the ring of integers $\mathcal{O}_K$ of $K$, and let $\chi$ be a Dirichlet character. Then define the twisted ideal divisor function $\sigma_{r, K, \chi} : I_K \rightarrow \mathbb{C}$ by $$\sigma_{r,K,\chi}(I) = \sum_{J \mid I} N(J)^{-r}\chi(N(J)),$$ where $N$ denotes the ideal norm. For real $r>1,$ we study the number of connected components $C_{r, K, \chi}$ of the closure $\overline{\sigma_{r,K,\chi}(I_K)}$, writing $C_{r,K}$ when $\chi$ is the principal character modulo 1. We prove that $C_{r,K,\chi}$ is finite when $\chi$ is real-valued. When $K = \mathbb{Q}$, we show that for fixed $r > 1,$ every sufficiently large positive integer is realized as $C_{r,\Q,\chi},$ and if $r$ is sufficiently large, then every positive integer is realized as $\chi$ varies. For finite Galois extensions $K$ over $\mathbb{Q}$, we exhibit new exponential lower bounds for $C_{r,K},$ and we prove that for every fixed integer $s \geq 2$, the values $C_{r,K}$ are unbounded as $K$ ranges over degree-$s$ extensions of $\mathbb{Q}$. 
\end{abstract}

\maketitle

\section{Introduction}

The objective of this paper is to understand the closure of the image of divisor functions. For a complex number $r$, the \textit{divisor function} $\sigma_r : \N \rightarrow \C$ is defined by $$\sigma_r(n) = \sum_{d \mid n} d^{-r},$$ where $\N = \Z_{> 0}.$ 
Since the mid-1980s, various authors have investigated the image $\sigma_r(\N).$ 
In 1986, Laatsch showed that $\sigma_1(\N)$ is dense in $[1, \infty)$ \cite{L86}. In 2000, Weiner asked what can be said about $\overline{\sigma_r(\N)},$ which is the closure of the image of $\sigma_r$ \cite{W00}. To this, in 2015 and 2016, Defant determined various fundamental topological properties of $\overline{\sigma_r(\N)}.$ In particular, he showed that $\overline{\sigma_r(\N)}$ is connected for $r \in \R$ if and only if $r \in (0, \eta],$ where $\eta$ is the unique real number satisfying $$\frac{2^\eta}{2^\eta - 1} \frac{3^\eta + 1}{3^\eta - 1} = \zeta(\eta)$$ and $\zeta$ is the Riemann zeta function \cite{D15}.
In 2017, Sanna proved that when $r > 1,$ the number $C_r$ is finite, giving an algorithm to compute $\overline{\sigma_r(\N)}$ \cite{S17}, and Defant showed that $C_r$ is finite when $\Re(r) > 1$ \cite{D17}.
In 2018, Zubrilina gave explicit lower and upper bounds on $C_r$ \cite{Z18}. In particular, she gave a lower bound of order $\pi(r),$ where $\pi(r)$ is the number of primes less than or equal to $r$, and an upper bound that is exponential in $r$. Additionally, she showed that $C_r$ does not assume all positive integral values.

In this paper, we introduce a generalized version of the divisor function, defined as follows, and study its properties.

\begin{definition}
    Let $r \in \C,$ and let $K$ be a finite extension of $\Q$. Let $R := \mathcal{O}_K$ be the ring of integers of $K$, and let $I_K$ be the set of integral ideals in $\mathcal{O}_K$. Let $\chi$ be a Dirichlet character.
    Then define the \textit{ideal divisor function on $K$ twisted by $\chi$} as the function $\sigma_{r,K,\chi} : I_K \rightarrow \C$ by $$\sigma_{r,K,\chi}(I) = \sum_{J \mid I} N(J)^{-r}\chi(N(J)),$$ where $N$ denotes the ideal norm. 
Define $C_{r,K,\chi}$ to be the number of connected components of $\overline{\sigma_{r,K,\chi}(I_K)}.$
\end{definition}

We now establish the notation we use throughout the paper. We fix a real number $r > 1$ and a finite  extension $K$ of $\Q.$ For $m \in \N$ and $x \in \R_{>0}$, let $\pi(x)$ denote the number of primes less than or equal to $n$ and $\pi_m(x)$ denote the number of primes less than or equal to $n$ that do not divide $m$. Denote the trivial character by $\chi_0$; note that the function $\sigma_{r, \Q, \chi_0}$ is precisely $\sigma_r.$ In Sections \ref{section-lower-2} and \ref{section-last}, we use $C_{r,K}$ to denote $C_{r,K,\chi_0}.$ In addition, we say a character $\chi$ is \textit{real} if it assumes values in $\{-1,0,1\}.$

This paper is structured as follows. 
In Section \ref{section-finiteness}, we prove that $C_{r,K,\chi}$ is finite, extending Sanna's proof that $C_{r}$ is finite in \cite[Theorem 1.1]{S17}.

\begin{theorem} \label{finiteness}
    Let $r > 1$ be a real number. let $K$ be a finite extension of $\Q,$ and let $\chi$ be a real character. Then $C_{r,K,\chi}$ is finite. 
\end{theorem} 

In Section \ref{section-range}, we prove that when $r$ is fixed and as $\chi$ varies, $C_{r,\Q,\chi}$ assumes all sufficiently large integer values; that is, 

\begin{theorem} \label{rangeC1}
    Let $r > 1$ be a real number. There exists $N \in \N$ such that $ \Z_{\geq N} \subseteq \{C_{r,\Q,\chi} : \chi \text{ is a character}\}.$
\end{theorem}

In particular, we show that when $r$ is large enough, $C_{r,\Q,\chi}$ assumes all positive integral values. 

\begin{theorem} \label{rangeC2}
    There exists a positive real number $r_0$ such that for any $r > r_0,$ we have $\{C_{r,\Q,\chi} : \chi \text{ is a character}\} = \N.$
\end{theorem}

In Section \ref{section-lower-1}, we give a lower bound on $C_{r,\Q,\chi}$ that improves upon the lower bound given by Achenjang and Berger in \cite[Theorem 5.2]{AB19}. 

\begin{theorem} \label{lowerbound1}
    Let $r \geq 3,$ and let $\chi$ be a character of modulus $m$. Then $C_{r,\Q,\chi} \geq 2^{\pi_m(2r-2)-\delta_m},$ where $\delta_m$ equals 0 if $m$ is even and equals 1 otherwise. 
\end{theorem}

Generalizing this result, in Section \ref{section-lower-2}, we give a lower bound on $C_{r,K}$ for finite Galois extensions $K$ of $\Q$.

\begin{theorem} \label{lowerbound2}
    Let $K/\Q$ be a finite Galois extension. For every $\varepsilon > 0,$ there exist constants $\mu_{K,\varepsilon}, r_{K,\varepsilon} > 0$ such that if $r > r_{K, \varepsilon},$ then $$\log C_{r,K} > \mu_{K, \varepsilon}\left( \frac{r^\frac{1}{1+\varepsilon}}{\log r} \right).$$
\end{theorem}

This result implies that fixing $K$, we have $\sup_{r > 1} C_{r,K} = \infty.$ As a complement to this result, in Section \ref{section-last}, we instead fix $s = [K:\Q]$ and $r$ and understand how $C_{r,K}$ behaves as $K$ changes.

\begin{theorem} \label{fix-s}
    Let $r > 1$ be a real number and $s \geq 2$ be an integer. Then $$\sup_{\text{degree-}s \text{ extensions }K/\Q} C_{r,K} = \infty.$$
\end{theorem}

We briefly illustrate some rough intuition for our results.
Below are the plots of the image of the divisor function $\sigma_{2, \Q, \chi}$ on the set $\{1,2,\dots,10^6\},$ where $\chi$ is the principal character $\chi_1$ mod 1 (the constant character) and $\chi_5$ mod 5, respectively.

\begin{figure}[H]
\centering
\includegraphics[scale=0.6]{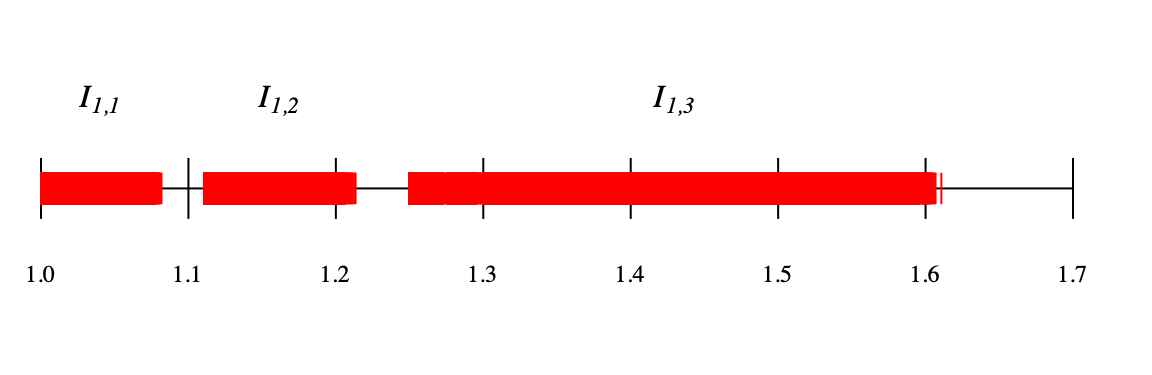}
\caption{The image $\sigma_{2,\Q,\chi_1}(\{1,2,\dots,10^6\}).$}
\includegraphics[scale=0.6]{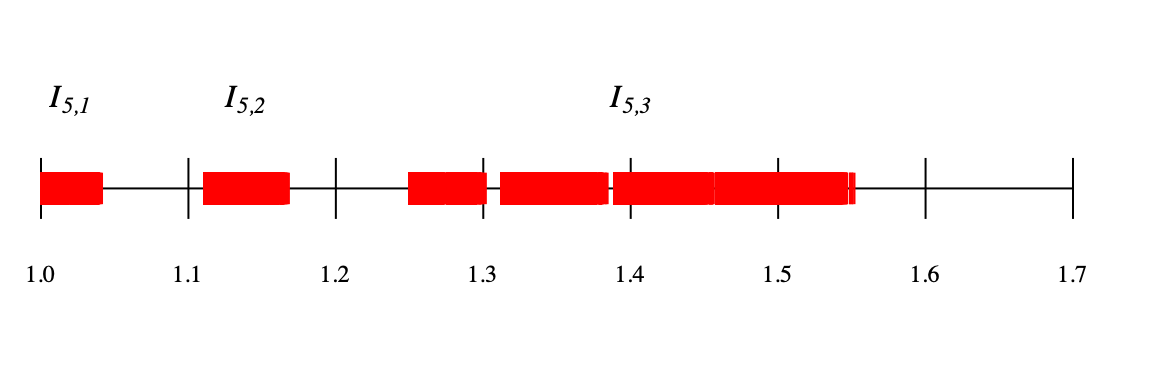}
\caption{The image $\sigma_{2,\Q,\chi_5}(\{1,2,\dots,10^6\}).$}
\end{figure}

As we can reasonably guess from these two figures, the numbers of connected components $C_{2,\Q,\chi}$ are finite. Indeed, they both have 3 connected components; let $I_{1,1}, I_{1,2},$ and $I_{1,3}$ be the connected components for $\chi_1$, and let $I_{5,1},I_{5,2},$ and $I_{5,3}$ be the connected components for $\chi_5.$ We notice that 
\begin{align*}
    I_{i,1} &= \sigma_{2,\Q,\chi_i}(\{n \in \Z_{>0} : 2 \nmid n, 3 \nmid n\})\\
    I_{i,2} &= \sigma_{2,\Q,\chi_i}(\{n \in \Z_{>0} : 2 \nmid n, 3 \mid n\}) \\
    I_{i,3} &= \sigma_{2,\Q,\chi_i}(\{n \in \Z_{>0} : 2 \mid n\})
\end{align*}
for $i = 1, 5.$
In this paper, we will prove this observation that these connected components are determined by the divisibility of the input integer by primes and use it to deduce results about $C_{r,K,\chi}$ in general.

\section{Finiteness of $C_{r,K,\chi}$ for $r > 1$ and real $\chi$} \label{section-finiteness}

The goal of this section is to prove Theorem \ref{finiteness}.

Throughout this section, let $\chi$ be a real character with modulus $m$. 
Order the prime ideals in $\mathcal{O}_K$ by increasing ideal norm; arbitrarily order prime ideals of the same norm. Denote the $i$-th prime ideal in this ordering by $\mathfrak{p}_i.$ 
Let $S_m$ be the finite set $\{i \in \N : \gcd(N(\p_i), m) \neq 1\}$ of indices of prime ideals whose norm are not coprime to $m$. Let $\mathcal{N}_j := \{ I \in I_K : \text{for all }1 \leq i \leq j, \p_i \nmid I\}$ denote the set of ideals whose prime ideal factors have index larger than $j$.

To prove the theorem, we prove a series of useful lemmas.


The following lemma establishes a density of norms of certain prime ideals in $\mathcal{O}_K$. Note that it holds for any choice of character $\chi$ with modulus $m$.

\begin{lemma} \label{adj-ideal}
    Consider the prime ideals $\p$ in $\mathcal{O}_K$ such that $\chi(N(\p)) = 1$, and let $\{\p_i'\}_{i \in \N}$ be the sequence of these ideals once ordered by increasing ideal norm. Then $$\lim_{i \rightarrow \infty} \frac{N(\p_{i+1}')}{N(\p_i')} = 1.$$
\end{lemma}

\begin{proof}
    Because there are finitely many prime ideals in $K$ that ramify in $K(\zeta_m),$ it is sufficient to prove the lemma for the sequence $\{\mathfrak p_i'\}_{i \in \N}$ of prime ideals $\p$ in $K$ that are unramified in $K(\zeta_m)$, satisfy $\chi(N(\p)) = 1,$ and are ordered by non-decreasing ideal norm $N$.
    
    Let $F = \Q(\zeta_m)$ and $L = K(\zeta_m).$ 
    Consider the isomorphism
    \begin{align*}
        \text{res} : \Gal(K(\zeta_m)/K) &\ra \Gal(\Q(\zeta_m)/\Q(\zeta_m) \cap K).
    \end{align*} Viewing $\Gal(\Q(\zeta_m)/\Q(\zeta_m) \cap K)$ as a subgroup of $\Gal(\Q(\zeta_m) / \Q) \cong (\Z/ m\Z)^\times,$ the Dirichlet character $\chi$ restricts to a character $$\Tilde{\chi} := \chi \circ \text{res}$$ of $\Gal(K(\zeta_m) / K).$
    For any prime ideal $\mathfrak{p}$ in $K$ that is unramified in $K(\zeta_m)/K$ and lies over $p\Z,$ the map $\res$ sends $\Frob_{\mathfrak{p}, K(\zeta_m)/K}$ to $$\Frob_{p, \Q(\zeta_m)/\Q}^{f_p} \in \Gal(\Q(\zeta_m) / \Q(\zeta_m) \cap K),$$ where $f_p$ is the inertia degree of $\p$ over $p\Z.$ 
    
    Now, let $H:= \ker \Tilde{\chi}$, which is a nonempty subgroup of $\Gal(K(\zeta_m)/K).$ For any $\sigma \in H,$ consider the prime ideals $\p$ in $K$ that are unramified in $K(\zeta_m)$ and satisfy $\Frob_{\p , K(\zeta_m)/K} = \sigma.$ Ordering these prime ideals by non-decreasing ideal norm $N$, we denote the resulting sequence by  $\{\p_{\sigma, i}\}_{i \in \N}.$ Because $\Tilde{\chi}(\Frob_{\p, K(\zeta_m)/K}) = \chi(N(\p)),$ we see that
    $$\{\p_i'\}_{ i \in \N} = \bigsqcup_{\sigma \in H} \{\p_{\sigma, i}\}_{i \in \N}.$$
    By Chebotarev's density theorem, we have $$\#\{\mathfrak p_{\sigma, i} : N(\mathfrak p_{\sigma,i}) \leq X\} \sim c_\sigma \frac{X}{\log X}$$ for some constant $c_\sigma > 0,$ so $$\lim_{i \ra \infty} \frac{N(\p_{\sigma, i+1})}{N(\p_{\sigma, i})} = 1.$$ As this holds for each $\sigma$ in the finite set $H$, it follows that $$\lim_{i \ra \infty} \frac{N(\p_{i+1}')}{N(\p_{i}')} = 1.$$
\end{proof}

\begin{definition}
    For $r > 1$, a nonnegative integer $i$, and a finite extension $K/\Q,$ define $$\zeta_{K,i}(r) = \prod_{k > i} \frac{1}{1-N(\p_k)^{-r}}.$$ For a positive integer $m$, further define $$\zeta_{K,i,m} (r) = \prod_{k>i; k\not\in S_m} \frac{1}{1-N(\p_k)^{-r}}.$$
\end{definition} 

Let $j_+ = \min\{j \in \mathbb N \setminus S_m : \chi(N(\mathfrak p_j)) = 1\}.$ 
For each integer $j \geq j_+$, let $\mathfrak{q}_j$ be a prime ideal with the largest norm less than or equal to $N(\mathfrak{p}_j)$ such that $\chi(N(\mathfrak{q}_j)) = 1.$  

\begin{lemma} [{\cite[Lemma 2.2]{S17}}] \label{finiteness-finitely-many-j}
    There exist only finitely many $j \in \N \setminus S_m$ satisfying $j \geq j_+$ such that $$\zeta_{K,j,m}(r) < 1 + N(\mathfrak{q}_j)^{-r}.$$
\end{lemma}

\begin{proof}
    Consider the sequence $(y_j)_{j \geq j_+}$ defined by $$y_j := \frac{\zeta_{K,j,m}(r)}{1 + N(\mathfrak{q}_j)^{-r}}.$$ For any $j \geq j_+,$ if $j + 1 \in S_m$, then $y_{j+1}=y_j.$ If $j + 1 \not\in S_m,$ then we compute 
    $$ \frac{y_{j+1}}{y_j} = (1 + N(\mathfrak{q}_j)^{-r}) \cdot \frac{1 - N(\p_{j+1})^{-r}}{1 + N(\mathfrak{q}_{j+1})^{-r}} < (1 + N(\mathfrak{q}_j)^{-r}) \cdot \frac{1 - N(\p_{j+1})^{-r}}{1 + N(\mathfrak{p}_{j+1})^{-r}},$$ because $N(\q_{j+1}) \leq N(\p_{j+1}).$ Now, let $\mathfrak{Q}_j$ be a prime ideal $\mathfrak{Q}$ with the least norm greater than $N(\q_j)$ such that $\chi(N(\QQ)) = 1.$ Then $N(\QQ) \geq N(\p_{j+1}).$ The function from $\R$ to $\R$ defined by $x \mapsto \frac{1-x}{1+x}$ is decreasing, so 
    $$\frac{y_{j+1}}{y_j} < (1 + N(\q_j)^{-r}) \cdot \frac{1 - N(\p_{j+1})^{-r}}{1 + N(\p_{j+1})^{-r}} < (1 + N(\q_j)^{-r}) \cdot \frac{1- N(\QQ_j)^{-r}}{ 1+ N(\QQ_j)^{-r}}.$$
    By Lemma \ref{adj-ideal}, we have $$\lim_{j\rightarrow \infty} \frac{N(\QQ_j)}{N(\q_j)} = 1.$$ 
    Then there exists $j_*$ such that if $j > j_*,$ then $\frac{N(\QQ_j)}{N(\q_j)} < 2^{1/r}$ and thus that $N(\QQ_j)^{-r} > (2 N(\q_j)^r + 1 )^{-1}.$ As a result, for $j > j_*,$ we find 
    $$\frac{y_{j+1}}{y_j} < (1 + N(\q_j)^{-r}) \cdot \frac{1 - (2 N(\q_j)^r + 1 )^{-1}}{1 + (2 N(\q_j)^r + 1 )^{-1}} = 1.$$  Therefore, the sequence $\{y_j\}_{j \geq j_+}$ is eventually non-increasing. Yet $\lim_{j \rightarrow \infty } y_j = 1,$ so $y_j \geq 1$ for all $j > j_*.$ The lemma's statement follows.
\end{proof}
The above lemma allows us to make the following definition.

\begin{definition} \label{defn-j}
    Define $j_{K,\chi}(r)$ to be the least element $j_0$ of $\N \setminus S_m$ such that $j_0 \geq j_+$, 
    \begin{itemize}
        \item[(1)] for all $j \in \N \setminus S_m$ with $j > j_0,$ we have $\zeta_{K,j,m}(r) \geq 1 + N(\mathfrak q_j)^{-r},$ and  
        \item[(2)] for all $j > j_0$ with $\chi(N(\mathfrak p_j)) = -1,$ we have $$ \frac{1}{1-N(\mathfrak p_j)^{-r}} < 1 + N(\mathfrak q_j)^{-r}.$$
    \end{itemize}
\end{definition}
We note that such an element $j_{K,\chi}(r)$ exists. By Lemma \ref{finiteness-finitely-many-j}, sufficiently large $j$ satisfy condition (1). When $\chi(N(\mathfrak p_j)) = -1,$ we have $N(\mathfrak q_j) < N(\mathfrak p_j),$ and the inequality in the second bulleted condition is equivalent to $N(\mathfrak q_j)^r + 1 < N(\mathfrak p_j)^r,$ which holds for sufficiently large $j$.


For a prime ideal $\p,$ define $$\sigma_{r,K,\chi}(\p^\infty) = \lim_{i \rightarrow \infty} \sigma_{r,K,\chi}(\p^i) = \frac{1}{1 - \chi(N(\mathfrak p)) N(\mathfrak p)^{-r}}.$$

\begin{lemma} [{\cite[Lemma 2.3]{S17}}] \label{finiteness-first-interval}
    For $r>1,$ we have $$\overline{\sigma_{r,K,\chi}( \mathcal{N}_{j_{K,\chi}(r)})} = \left[ \prod_{k > j_{K,\chi}(r); \chi(N(\p_k)) = -1} 1 - N(\p_k)^{-r}, \prod_{k > j_{K,\chi}(r); \chi(N(\p_k)) = 1} \frac{1}{1-N(\p_k)^{-r}}\right].$$
\end{lemma}

\begin{proof}
    Denote the interval above by $[c,d].$
    Take $x \in [c,d].$ We define a sequence $(x_j)_{j \geq j_{K,\chi}(r)}$ as follows. Let $x_{j_{K,\chi}(r)} = c.$ For any $j > j_{K,\chi}(r),$ we define $x_j$ using the following process.
    If $\chi(N(\p_j)) = 0,$ then set $x_j = x_{j-1}.$
    If $\chi(N(\p_j)) = 1,$ then let $a_j$ be the largest nonnegative integer $a$ (including $\infty$) such that $x_{j-1} \cdot \sigma_{r,K,\chi}(\p_j^a) \leq x,$ and let $$x_j = x_{j-1} \cdot \sigma_{r,K,\chi}(\p_j^{a_j}).$$ If $\chi(N(\p_j)) = -1,$ then if $x_{j-1} \cdot \frac{1}{1-N(\p_j)^{-r}} \leq x$, let $b_j = 1$, and let $$x_j = x_{j-1} \cdot \frac{1}{1-N(\p_j)^{-r}}.$$ Otherwise, let $b_j=0$, and let $x_j=x_{j-1}.$ Since the sequence $(x_j)_{j\geq j_{K,\chi}(r)}$ is non-decreasing and bounded, it converges, and let its limit be $\ell.$ Note that $\ell \leq x$. 
    If there are infinitely many $j > j_{K,\chi}(r)$ with $\chi(N(\p_j)) = 1$ and $a_j < \infty$, then $x < x_{j-1} \cdot \frac{1}{1 - N(\p_j)^{-r}}$ for infinitely many $j > j_{K,\chi}(r),$ so $x \leq \ell$ and thus $x = \ell.$
    If there are infinitely many $j >j_{K,\chi}(r)$ with $\chi(N(\p_j)) = -1$ and $b_j=0$, then a similar argument gives $x =\ell.$ Otherwise, there must only be finitely many $j > j_{K,\chi}(r)$ for which $a_j < \infty$ and only finitely many $j > j_{K,\chi}(r)$ for which $b_j = 0.$ 
    Let $j'$ be the least integer greater than equal to $j_{K,\chi}(r)$ such that for each $j > j'$ 
    with $j \not\in S_m,$ the following two conditions hold: when $\chi(N(\p_j)) = 1,$ we have $a_j=\infty,$ and when   $\chi(N(\p_j)) = -1$, we have $b_j = 1.$ Then for any $j > j',$ we have $$x_j = x_{j'} \cdot \prod_{j' < k \leq j; k \not\in S_m} \frac{1}{1-N(\p_k)^{-r}}$$ and thus that $$\ell = x_{j'} \cdot \prod_{k > j'; k \not\in S_m} \frac{1}{1-N(\p_k)^{-r}}.$$ If $j' = j_{K,\chi}(r),$ then $$\ell = c \cdot \prod_{k > j_{K,\chi}(r); k \not\in S_m} \frac{1}{1-N(\p_k)^{-r}} = d$$ and $\ell \leq x \leq d,$ so $\ell = x.$
    Otherwise, if $j' > j_{K,\chi}(r),$ then by the minimality of $j'$, we have $j' \not\in S_m.$ Thus, either $\chi(N(\p_{j'})) = 1$ and $a_{j'} < \infty,$ or $\chi(N(\p_{j'})) = -1$ and $b_{j'} = 0.$ 
    If $a_{j'} <\infty,$ then \begin{align*}
        x_{j'} \cdot \prod_{k > j'; k \not\in S_m} \frac{1}{1-N(\p_k)^{-r}} = \ell \leq x &< x_{j'-1} \cdot \sigma_{r,K,\chi}(\p_{j'}^{a_{j'} + 1}) \\
        &\leq \frac{x_{j'}}{\sigma_{r,K,\chi}(\p_{j'}^{a_{j'}})} \cdot  \sigma_{r,K,\chi}(\p_{j'}^{a_{j'} + 1}) \\
        &\leq x_{j'}  (1 + N(\p_{j'})^{-r})\\
        &= x_{j'}(1 + N(\q_{j'})^{-r}),
    \end{align*} where the last equality holds because $\chi(N(\p_{j'})) = 1.$ This contradicts condition (1) in the definition of $j_{K,\chi}(r).$ Similarly, if $b_{j'} = 0,$ then \begin{align*}
        x_{j'} \cdot \prod_{k > j'; k \not\in S_m} \frac{1}{1-N(\p_k)^{-r}} = \ell \leq x &< x_{j'-1} \cdot \frac{1}{1-N(\p_{j'})^{-r}} \\
        &= x_{j'} \cdot \frac{1}{1-N(\p_{j'})^{-r}} \\
        &< x_{j'} (1 + N(\q_{j'})^{-r}),
    \end{align*} where the last inequality holds due to the condition (2) in the definition of $j_{K,\chi}(r).$ This contradicts condition (1) in the definition of $j_{K,\chi}(r) \leq j'.$ Hence, $j' =j_{K,\chi}(r),$ so $x = \ell.$
\end{proof}

\begin{lemma} [{\cite[Lemma 2.4]{S17}}] \label{finiteness-induct}
    For any $1 \leq j \leq j_{K,\chi}(r),$  $$\overline{\sigma_{r,K,\chi}(\mathcal{N}_{j-1})} = \bigcup_{a \in \Z_{\geq 0} \cup \{\infty\}} \sigma_{r,K,\chi}(\p_j^a) \cdot \overline{\sigma_{r,K,\chi}(\mathcal{N}_j)}.$$
\end{lemma}

\begin{proof}
    It is clear that the right-hand side is contained in the left-hand side. To see the reverse containment, take $x \in \overline{\sigma_{r,K,\chi}(\mathcal{N}_{j-1})}.$ Then there exists a sequence $\{I_i\}_{i \in \N} \subset \mathcal{N}_{j-1}$ such that $\lim_{i \rightarrow \infty} \sigma_{r,K,\chi}(I_i) = x.$ For each $i > 0,$ let $a_i$ be the largest power of $\p_j$ dividing $I_i.$ If $\{a_i\}_{i \in \N}$ is bounded, then there exists a nonnegative integer $a_0$ that is attained infinitely many times in this sequence; let $\{i_\ell\}_{\ell \in \N}$ be the corresponding subindices. Then $$\frac{x}{\sigma_{r,K,\chi}(\p_j^{a_0})} = \lim_{\ell \rightarrow \infty} \sigma_{r,K,\chi}\left( I_{i_\ell} \p_j^{-a_0}\right) \in \overline{\sigma_{r,K,\chi}(\mathcal{N}_j)}.$$ If $\{a_i\}_{i \in \N}$ is unbounded, there exists a subsequence $\{a_{i_\ell}\}_{\ell \in \N}$ that converges to $\infty.$ Then $$\frac{x}{\sigma_{r,K,\chi}(\p_j^\infty)} = \lim_{\ell \rightarrow \infty} \sigma_{r,K,\chi}(I_{i_\ell} \p_j^{-a_{i_\ell}}) \in \overline{\sigma_{r,K,\chi}(\mathcal{N}_j)}.$$
\end{proof}

\begin{lemma} [{\cite[Lemma 2.5]{S17}}] \label{finiteness-last}
    Let $j > 0,$ and let $0 < \alpha < \beta.$  The set $$\bigcup_{a \in \Z_{\geq 0} \cup \{\infty\}} \sigma_{r,K,\chi}(\p_j^a) \cdot [\alpha,\beta]$$ is a finite 
    union of closed bounded intervals. In particular, if $\chi(N(\mathfrak p_j)) = 1$ and  $a_0$ is the least nonnegative integer such that $$\frac{\sigma_{r,K,\chi}(\p_j^{a_0+1})}{\sigma_{r,K,\chi}(\p_j^{a_0})} \leq \frac{\beta}{\alpha},$$ then the set above is the finite union of precisely $a_0+1$ pairwise disjoint closed bounded intervals. 
\end{lemma}

\begin{proof}
    We consider the three possible cases of the value of $\chi(N(\mathfrak p_j)).$ If $\chi(N(\mathfrak p_j)) = 0,$ then the desired union is simply $[\alpha,\beta].$ Next, suppose $\chi(N(\mathfrak p_j)) \neq 0.$ Because $\sigma_{r,K,\chi}(\mathfrak p_j^a)$ converges to $\sigma_{r,K,\chi}(\mathfrak p_j^\infty)$ as $a$ tends to $\infty,$ there exists $A$ for which $a \geq A$ implies the intervals $\sigma_{r,K,\chi}(\mathfrak p_j^a)[\alpha,\beta]$ and $\sigma_{r,K,\chi}(\mathfrak p_j^{a+1})[\alpha,\beta]$ intersect. The total union can thus be express as the union of at most $A+1$ pairwise disjoint closed bounded intervals. In particular, when $\chi(N(\mathfrak p_j)) = 1,$ if we take $A$ to be $a_0,$ then the total union consists of precisely $a_0+1$ pairwise disjoint closed bounded intervals. 
\end{proof}

We are ready to prove Theorem \ref{finiteness}.

\begin{proof}[Proof of Theorem \ref{finiteness}]
    By Lemma \ref{finiteness-first-interval}, $\overline{\sigma_{r,K,\chi}(\mathcal{N}_{j_{K,\chi}(r)})}$ is a single bounded closed interval. We use downward induction to show that for each $1 \leq j \leq j_{K,\chi}(r),$ the closure $\overline{\sigma_{r,K,\chi}(\mathcal{N}_{j-1})}$ is a finite union of pairwise disjoint closed bounded intervals. By Lemmas \ref{finiteness-induct} and \ref{finiteness-last}, the claim holds for $j = j_{K,\chi}(r).$ Suppose the claim holds for some $j = j'$; that is, $\overline{\sigma_{r,K,\chi}(\mathcal{N}_{j'})} = \bigsqcup_{t \in T} C_t$, where $T$ is finite and each $C_t$ is a bounded closed interval. By Lemma \ref{finiteness-induct}, we have $$\overline{\sigma_{r,K,\chi}(\mathcal{N}_{j'-1})} = \bigcup_{t \in T} \left( \bigcup_{a \in \Z_{\geq 0} \cup \{\infty\}} \sigma_{r,K,\chi}(\p_{j'}^a) \cdot C_t \right).$$ For each $t \in T,$ the inner union is a finite union of closed bounded intervals by Lemma \ref{finiteness-last}, so $\overline{\sigma_{r,K,\chi}(\mathcal{N}_{j'-1})}$ is the finite union of closed bounded intervals. Since $\mathcal{N}_0 = I_K,$ the theorem is proved.
\end{proof}

\section{Range of $C_{r,\Q,\chi}$ when $r$ is fixed} \label{section-range}

In this section, we prove some results about the values of $C_{r,\Q,\chi}$ for $r > 1$ and a Dirichlet character $\chi.$ In \cite[Theorem 3.3]{D16}, Defant shows that for $r \in \C$ with $\Re(r) > 1,$ we have $\lim_{r \rightarrow \infty} C_{r,\Q,\chi_0} = \infty.$ One can deduce from the theorem's proof that for a fixed character $\chi,$ we have $\lim_{r \rightarrow \infty} C_{r,\Q,\chi} = \infty.$ In this vein, we fix $r>1$ and determine what values $C_{r,\Q,\chi}$ can attain for arbitrary real character $\chi.$ 

In particular, we will prove Theorem \ref{rangeC1}, which asserts that for any $r > 1,$ there exists large enough $N \in \mathbb{N}$ for which $C_{r,\Q,\chi}$ can attain any integral value at least $N$. 
The key idea of the proof is that increasing the number and the magnitude of the distinct primes dividing a character's modulus forces additional jumps in the divisor function's range, thus generally increasing the number of connected components in the range. Thus, the most natural collection of characters we can use to witness and control these effects in the range are principal characters, which are determined entirely by the set of distinct prime factors of their moduli.

In this section, we use notation in a similar fashion to that of the previous section. Fix $r > 1.$ Let $p_j$ be the $j$-th rational prime and $\mathcal{N}_j$ be the set of positive integers not divisible by the first $j$ primes. For any positive integer $m$, let $\chi_{m}$ be the principal modulus-$m$ character. Lastly, recalling Definition \ref{defn-j}, set $j_m:= j_{\Q,\chi_m}(r).$

We proceed to define some new notation to select the characters $\chi$ in order for $C_{r,\Q,\chi}$ to attain any sufficiently large integral value. We define $\ell := \max(2, j_1)$ and 
$$ i_0= \begin{cases}
    j_1 & \text{if }j_1 \geq 2 \\
    \min\bigl\{i\in \mathbb{Z}_{>0} : \zeta_{\Q, i}(r) < 1+3^{-r}\} & \text{else.}
\end{cases}$$
Note that $i_0 \geq \ell,$ so we finally define the sequence of positive integers $(m_i)_{i \geq i_0}$ by $$m_i := \frac{1}{p_\ell} \prod_{p \leq p_i} p.$$ The moduli $m_i$ are chosen so that $p_\ell$ is the unique prime that contributes Euler factors to the image of $\sigma_{r,\Q,\chi}$. In particular, the prime $p_\ell$ creates ``copies" of $[1, \zeta_{\Q,i}(r)]$ in $\overline{\sigma_{r,\Q,\chi}(\N)}$. Increasing $i$ shrinks the lengths of these interval copies and thus increases the number of connected components.

\begin{lemma} \label{rangeClemma}
    Let $i \geq i_0.$ Then $(C_{r,\Q,\chi_{m_i}})_{i \geq i_0}$ is a non-decreasing sequence that is not eventually constant. In particular, for any $i \geq i_0,$ we have $$C_{r,\Q,\chi_{m_i}} = \left \lceil \log_{p_\ell^{-r}} \frac{\zeta_{\Q,i}(r)-1}{\zeta_{\Q,i}(r) - p_\ell^{-r}} \right \rceil.$$
\end{lemma}

\begin{proof}
    Let $i \geq i_0.$ We first show that $j_{m_i} = \ell.$ It is clear that $j_+ = \ell.$ Now, take any $j \in \mathbb N \setminus S_{m_i}$ for which $j > \ell.$ Because $m_i$ is divisible by each $p_k$ with $k \leq i$ except $p_\ell,$ this implies $j > i.$ Hence, $\zeta_{\Q, j , m_i}(r) = \zeta_{\Q,j}(r).$ For any $j  > i,$ we have $\zeta_{\Q,j}(r) \geq 1 + p_j^{-r}$ because $i \geq i_0.$ Thus, $\ell$ satisfies the defining property of $j_{m_i}.$ 

    Since $\sigma_{r,\Q,\chi_{m_i}} (p_j^a) = 1$ for $1 \leq j \leq i$ with $j \neq \ell,$ we apply Lemma \ref{finiteness-induct} for $1 \leq j \leq \ell$ and subsequently apply Lemma \ref{finiteness-first-interval}
    to obtain \begin{align*}
        \overline{\sigma_{r, \Q,\chi_{m_i}} (\N)} &= \overline{\sigma_{r, \Q,\chi_{m_i}} (\mathcal{N}_1)} \\
        &= \bigcup_{a \in \Z_{\geq 0} \cup\{\infty\}} \sigma_{r, \Q,\chi_{m_i}} (p_\ell^a) \cdot \overline{\sigma_{r, \Q,\chi_{m_i}} (\mathcal{N}_{\ell})} \\
        &= \bigcup_{a \in \Z_{\geq 0} \cup\{\infty\}} \sigma_{r, \Q,\chi_{m_i}} (p_\ell^a) \cdot [1, \zeta_{\Q, j_{m_i}, m_i}(r)].
    \end{align*} Because $j_{m_i} = \ell,$ we have $\zeta_{\Q, j_{m_i}, m_i}(r) = \zeta_{\Q, \ell, m_i}(r).$ As each prime $p_k$ with $\ell < k \leq i$ divides $m_i,$ we have $\zeta_{\Q,\ell,m_i}(r) = \zeta_{\Q,i}(r).$ Thus, we have $$\overline{\sigma_{r, \Q,\chi_{m_i}} (\N)} = \bigcup_{a \in \Z_{\geq 0} \cup \{\infty\}}  \sigma_{r, \Q,\chi_{m_i}} (p_\ell^a) \cdot [1, \zeta_{\Q, i}(r)].$$

    Now, let $a_i$ be the least nonnegative integer such that \begin{equation} \label{eq1}
        \frac{ \sigma_{r,\Q,\chi_{m_i}} (p_{\ell}^{a_i+1}) }{ \sigma_{r,\Q,\chi_{m_i}} (p_{\ell}^{a_i}) } \leq \zeta_{\Q,i}(r).
    \end{equation} By Lemma \ref{finiteness-last}, we have $C_{r,\Q,\chi_{m_i}} = a_i+1.$
    Note that the left hand side of the equation is $$\frac{ 1+p_{\ell}^{-r} + \cdots + p_{\ell}^{-(a_i+1)r} }{  1+p_{\ell}^{-r} + \cdots + p_{\ell}^{-a_i r}  } = \frac{1 - (p_{\ell}^{-r})^{a_i+2}}{1 - (p_{\ell}^{-r})^{a_i+1}}.$$ Algebraic manipulation shows that $$C_{r,\Q,\chi_{m_i}} = a_i +1 = \left\lceil \log_{p_\ell^{-r}} \frac{\zeta_{\Q,i}(r) - 1}{\zeta_{\Q,i}(r) - p_\ell^{-r}} -1 \right\rceil+1 = \left\lceil \log_{p_\ell^{-r}} \frac{\zeta_{\Q,i}(r) - 1}{\zeta_{\Q,i}(r) - p_\ell^{-r}} \right\rceil.$$ As $i$ increases, the expression inside the ceiling increases and approaches $\infty$ because $\zeta_{\Q,i}(r)$ approaches 1 from the right. Thus, $(C_{r,\Q,\chi_{m_i}})_{i \geq i_0}$ is non-decreasing and not eventually constant.
\end{proof}

We proceed to prove Theorem \ref{rangeC1}. 

\begin{proof}[Proof of Theorem \ref{rangeC1}]
    For each $i$, define $y_i = \zeta_{\Q,i}(r).$
    By Lemma \ref{rangeClemma}, it suffices to prove that $C_{r,\Q,\chi_{m_{i+1}}} - C_{r,\Q,\chi_{m_i}} \leq 1$ for any $ i \geq i_0.$  Recall from Lemma \ref{rangeClemma} that $C_{r,\Q,\chi_{m_i}} = a_i+1,$ where $a_i$ is the least nonnegative integer satisfying Equation (\ref{eq1}).
    We essentially show that the sequence $$\left(\frac{1 - (p_\ell^{-r})^{k+2}}{1 - (p_\ell^{-r})^{k+1}} \right)_{k \geq 0}$$ converges to 1 at a faster rate than $(\zeta_{\Q,i}(r))_{i \geq i_0}$ does. 
    In particular, we prove that for any $k \geq0$ and $i \geq i_0,$ we have $$\frac{\frac{1 - (p_\ell^{-r})^{k+3}}{1 - (p_\ell^{-r})^{k+2}}  - 1}{\frac{1 - (p_\ell^{-r})^{k+2}}{1 - (p_\ell^{-r})^{k+1}}  - 1} < \frac{\zeta_{\Q,i+1}(r) - 1}{\zeta_{\Q,i}(r)- 1}.$$ 
    This proves the desired, because if $C_{r,\Q, \chi_{m_{i+1}}} - C_{r,\Q, \chi_{m_i}} > 1$ for some $i \geq i_0,$ then $a_{i+1} \geq a_i+ 2.$ This would imply that there exists $k \geq 0$ such that $$\frac{1 - (p_\ell^{-r})^{k+2}}{1 - (p_\ell^{-r})^{k+1}} \leq y_{i}\text{ and } \frac{1 - (p_\ell^{-r})^{k+3}}{1 - (p_\ell^{-r})^{k+2}} > y_{i+1},$$ which would contradict the claim. 
    
    For any $k \geq 0,$ we find that \begin{align*}
        \frac{ \frac{1 - (p_\ell^{-r})^{k+3}}{1 - (p_\ell^{-r})^{k+2}} - 1}{\frac{1 - (p_\ell^{-r})^{k+2}}{1 - (p_\ell^{-r})^{k+1}} - 1} 
        &= p_\ell^{-r} \cdot \frac{1 - (p_\ell^{-r})^{k+1} }{1 - (p_\ell^{-r})^{k+2}} \\
        &< 3^{-r}
    \end{align*} because $\ell \geq 2.$
    On the other hand, for any $i \geq i_0,$ we find that 
    \begin{align*}
        \frac{\zeta_{\Q,i+1}(r) - \zeta_{\Q,i+2}(r)}{\zeta_{\Q,i}(r) - \zeta_{\Q,i+1}(r)} &= \frac{ \left( \frac{1}{1-p_{i+2}^{-r}} -  1 \right) \zeta_{\Q,i+2}(r) }{  \left( \frac{1}{1-p_{i+1}^{-r}} -  1 \right) \zeta_{\Q,i+1}(r)  }\\
        &= \frac{p_{i+2}^{-r} (1 - p_{i+1}^{-r})}{p_{i+1}^{-r}} \\
        &> 2^{-r}(1 - 3^{-r}),
    \end{align*}
    where the last step follows from Bertrand's postulate and the fact that $p_{i+1} \geq p_2 = 3$. Inductively, we conclude that for any $k \geq 0,$ we have $$\zeta_{\Q,i+k}(r) - \zeta_{\Q,i+k+1}(r) > (2^{-r} (1 - 3^{-r}) )^k (\zeta_{\Q,i}(r) - \zeta_{\Q,i+1}(r)) > (3^{-r} )^k (\zeta_{\Q,i}(r) - \zeta_{\Q,i+1}(r)).$$ As a result, \begin{align*}
        \zeta_{\Q,i+1}(r) - 1 &= \sum_{k=1}^\infty \zeta_{\Q,i+k}(r) - \zeta_{\Q,i+k+1}(r) \\&> (3^{-r} + (3^{-r})^2 + \cdots ) (\zeta_{\Q,i}(r) - \zeta_{\Q,i+1}(r)) \\ &= \frac{1}{3^r - 1} (\zeta_{\Q,i}(r) - \zeta_{\Q,i+1}(r)),
    \end{align*}so we find that $$\frac{\zeta_{\Q,i+1}(r) - 1}{\zeta_{\Q,i}(r)- 1} > \frac{ \frac{1}{3^r - 1} }{ \frac{1}{3^r - 1} + 1} = 3^{-r}.$$ 
    Then setting $N$ in the theorem statement to be $C_{r,\Q,\chi_{m_{i_0}}}$ gives the theorem.
\end{proof}

We proceed to prove Theorem \ref{rangeC2}, where we show that when $r$ is large enough, we can describe $C_{r,\Q,\chi}$ in greater detail.

\begin{proof}[Proof of Theorem \ref{rangeC2}]
    We first show the existence of a real number $r_0 > 1$ such that $$\sum_{n \geq 3} n^{-r} < 2^{-r} \text{ and } \sum_{n\geq 4} n^{-r} < 3^{-r}$$ hold for any $r > r_0.$ Indeed, we can bound $$\sum_{n \geq 3} n^{-r} \leq 3^{-r} + \int_3^\infty x^{-r} \; dx = 3^{-r} \left(1 + \frac{3}{r-1} \right),$$ which is less than $2^{-r}$ for sufficiently large $r$. By the same reasoning, choosing $r$ to be large enough ensures the second inequality. 

    Now, let $r > r_0.$ Then $$\zeta(r) = 1 + 2^{-r} + \sum_{n \geq 3} n^{-r} < 1 + 2 \cdot 2^{-r} < \frac{1+2^{-r}}{1-2^{-r}}$$ and hence $\zeta_{\Q,1}(r) = \zeta(r)(1-2^{-r}) < 1 + 2^{-r}.$ Similarly, we have $$\zeta(r) = 1 + 2^{-r} + 3^{-r} + \sum_{n \geq 4} n^{-r} < 1 + 2^{-r} + 2 \cdot 3^{-r} < \frac{1+3^{-r}}{(1-2^{-r})(1-3^{-r})}$$ and hence $\zeta_{\Q,2}(r) = \zeta(r)(1-2^{-r})(1-3^{-r}) < 1+3^{-r}.$ Consequently, $j_1 \geq 2,$  so $\ell = \max(2,j_1) = j_1.$


    Now, we show that $C_{r,\Q,\chi_{m_{j_1}}} = 2$.  By Lemma \ref{rangeClemma}, we wish to show that $$1 < \log_{p_{j_1}^{-r}} \frac{\zeta_{\Q,j_1}(r)-1}{\zeta_{\Q,j_1}(r) - p_{j_1}^{-r}} < 2.$$ This is equivalent to showing that $\zeta_{\Q,j_1}(r) < 1 + p_{j_1}^{-r}$ and $$\zeta_{\Q,j_1}(r) > \frac{(p_{j_1}^{-r})^2 + p_{j_1}^{-r} + 1}{p_{j_1}^{-r} + 1}.$$ The first inequality follows from the minimality of $j_1$. We proceed to check the second inequality. Let $q = p_{j_1}$ and let $q_* = p_{j_1+1}$. By the definition of $j_1,$ we have $$\prod_{p > q} \frac{1}{1-p^{-r}} > \frac{1 + q_*^{-r}}{1 - q_*^{-r}}.$$ The right hand side of the inequality is a decreasing function in $q_*$, and Bertrand's Postulate tells us that $q_* < 2q,$ so the desired follows from showing $$\frac{1+q^{-r} + (q^{-r})^2}{1 + q^{-r}} < \frac{1 + (2q)^{-r}}{1 - (2q)^{-r}}.$$ Let $x = q^{-r}.$ The inequality above equates to $x^2 + (2-2^r)x + 2 > 0,$ which we need only prove holds for $x \leq 2^{-r}.$ If $1 < r < 2,$ then $2 - 2^r > -2,$ so $x^2 + (2-2^r)x + 2 > x^2 - 2x+ 2 = (x-1)^2 + 1 > 0.$ Now, suppose $r \geq 2.$ Note that $x^2 + (2-2^r)x + 2$ is decreasing on $x \in (0,2^{-r}]$ because its derivative is $2x + 2 - 2^r < 0.$ Therefore, $x^2 + (2-2^r)x + 2 \geq (2^{-r})^2 + (2-2^r)(2^{-r}) + 2 > 0,$ so the inequality holds. 

    Using Lemma \ref{rangeClemma}, we have now shown that for any $r > r_0,$ any integer greater than 1 occurs as $C_{r,\Q,\chi}$ for some character $\chi$. To conclude, we note that there exists a character $\chi$ such that $C_{r,\Q,\chi} = 1.$ Let $$m = \prod_{k \leq j_1} p_k.$$ Because $\chi_m(p_k)=0$ for $k\leq j_1,$ we have $\overline{\sigma_{r,\Q,\chi_m} ( \N)} = \overline{\sigma_{r,\Q,\chi_0}(\mathcal N_{j_1})}.$ As $j_1 = j_{\Q,\chi_0}(r),$ we then apply Lemma \ref{finiteness-first-interval} to conclude $\overline{\sigma_{r,\Q,\chi_m} ( \N)}= [ 1, \zeta_{\Q,j_1}(r)]$ and thus $C_{r,\Q,\chi_m} =1.$
\end{proof}

\section{A lower bound for $C_{r,\Q,\chi}$} \label{section-lower-1}

In this section, we prove Theorem \ref{lowerbound1}, which gives a lower bound on $C_{r,\Q,\chi}$ for any $r \geq 3$ and any character $\chi$. This result improves upon \cite[Theorem 5.2]{AB19}, which gives a lower bound on $C_{r,\Q}$ for any $r > 1$. Inspired by the proof of \cite[Theorem 5.2]{AB19}, the key mechanism behind the proof of Theorem \ref{lowerbound1} is to partition the natural numbers based on prime factors bounded by functions of $r$ and the modulus of $\chi$, which then capture explicit gaps in the range of the associated divisor function. 
To begin, we prove the following useful lemma. 

\begin{lemma} \label{tail}
    Let $r > 1.$
    \begin{itemize}
        \item[(1)] For any integer $0 < d \leq r-1,$ we have $$d^{-r} > \sum_{n = d+1}^\infty n^{-r}.$$
        \item[(2)] For any odd integer $d \leq 2r-2,$ we have $$d^{-r} > \sum_{n \geq d+2; 2 \nmid n} n^{-r}.$$ 
    \end{itemize}
\end{lemma}

\begin{proof}
    (1) is proved in \cite[Lemma 5.1]{AB19}. (2) follows from noting that $$\sum_{\text{odd } n = d+2}^\infty n^{-r} < \int_{\frac{d-1}{2}}^\infty (2x+1)^{-r} dx = \frac{d^{-r+1}}{2(r-1)},$$ which is less than or equal to $d^{-r}$ when $d \leq 2r-2.$
\end{proof}
We are ready to prove Theorem \ref{lowerbound1}.

\begin{proof}[Proof of Theorem \ref{lowerbound1}]
Let $\theta$ be a nonzero complex number. Then any complex number $x \in \C$ can be written uniquely in the form $a \theta + b \theta i,$ where $a,b \in \R.$ Define $\Re_\theta$ by $\Re_\theta(a\theta + b \theta i) = a.$

We begin by partitioning $\N$. To do so, we first define the map $B_1: \N \rightarrow \mathcal{P}(\N)$ by $B_1(x) = \{d \in \N : d \mid x, d \leq r-1, \gcd(d,m)=1\}.$ Next, we define the map $B_2 : \N \rightarrow \mathcal {P} (\N)$ by $$B_2(n) = \begin{cases}
    \{d \in \N : d \mid n, r-1 \leq d \leq 2r-2, \gcd(d, 2m) = 1\} & \text{ if } 2 \not\in B_1(n) \\
    \emptyset &\text{ if }2 \in B_1(n).
\end{cases}$$
Finally, define $A(n) = (B_1(n), B_2(n)).$
Note that this determines a partition of $\N$ whose parts are $A^{-1}(S)$ for $S \in A(\N).$ We wish to show that for any distinct $S_1, S_2 \in A(\N),$ the sets $\sigma_{r,\Q,\chi}(A^{-1}(S_1))$ and $\sigma_{r,\Q,\chi}(A^{-1}(S_2))$ have disjoint closures. Write $S_1 = A(x)$ and $S_2 = A(y)$ for distinct $x,y \in \N.$ 

    We consider separately the cases $B_1(x) \neq B_1(y)$ and $B_1(x) = B_1(y)$. Assume first that $B_1(x) \neq B_1(y)$.  Suppose without loss of generality that the smallest integer $d_0 \leq r-1$ coprime to $m$ that divides precisely one of $x$ and $y$ lies in $B_1(x)$. Let $\theta = \chi(d_0).$ Take $x'$ such that $B_1(x') = B_1(x)$ and $y'$ such that $B_1(y') = B_1(y).$ For each $d < d_0,$ the difference of the terms for $d$ in $\sigma_{r,\Q,\chi} (x') - \sigma_{r, \Q,\chi}(y')$ is 0: if $\gcd(d,m) > 1,$ then $\chi(d) = 0,$ and otherwise, $d \in B_1(x')$ if and only if $d \in B_1(y')$ by the minimality of $d_0.$ Then we have $$\sigma_{r,\Q,\chi} (x') - \sigma_{r, \Q,\chi}(y') = \theta d_0^{-r} + \sum_{d \mid x'; d > d_0} \chi(d)  d^{-r} - \sum_{d \mid y'; d > d_0} \chi(d) d^{-r}.$$ Applying $\Re_\theta$ to both sides, we obtain $$\Re_\theta(\sigma_{r,\Q,\chi}(x')) - \Re_\theta(\sigma_{r,\Q,\chi}(y')) \geq d_0^{-r} - \sum_{d > d_0} d^{-r} > 0$$ by Lemma \ref{tail}. 
Thus, the closures of $\sigma_{r,\Q,\chi}(A^{-1}(A(x))$ and $\sigma_{r,\Q,\chi}(A^{-1}(A(y))$ are disjoint. 

    It remains to consider the case $B_1(x) = B_1(y).$ Then $B_2(x) \neq B_2(y),$ so one of these sets is nonempty, meaning $2 \not\in B_1(x) = B_1(y).$ If $m$ is even, then $\chi(d) = 0$ for even $d$. If $m$ is odd, because  $r-1 \geq 2$, the fact that $2 \not\in B_1(x) = B_1(y)$ implies that $x$ and $y$ have no even divisors.  The argument then proceeds in a similar fashion to the case $B_1(x) \neq B_1(y).$ Let $d_0$ be the smallest integer that is in precisely one of the two sets $B_2(x)$ and $B_2(y).$ Without loss of generality, we assume $d_0 \in B_2(x) \setminus B_2(y).$ Let $\theta = \chi(d_0),$ and take $x'$ such that $A(x') = A(x)$ and $y'$ such that $A(y') = A(y).$ When comparing $\Re_\theta(\sigma_{r,\Q,\chi}(x'))$ and $\Re_\theta(\sigma_{r,\Q,\chi}(y'))$, we use the fact that the contributions from divisors $d < d_0$ are the same. In addition, there are no contributions from even divisors. Indeed, if $m$ is even, then $\chi(d) = 0$ for even $d$; if $m$ is odd, because  $r-1 \geq 2$, the fact that $2 \not\in B_1(x) = B_1(y)$ implies that $x$ and $y$ have no even divisors.  More explicitly, we compute 
    \begin{align*}
        \Re_\theta(\sigma_{r,\Q,\chi}(x'))-\Re_\theta(\sigma_{r,\Q,\chi}(y')) &= \Re_\theta \left( \sum_{d \mid x'; d < d_0} d^{-r} \chi(d) + d_0^{-r} \theta + \sum_{d \mid x'; d \geq d_0+2; d \text{ odd}} d^{-r}\chi(d)\right) \\
        &\text{ \; } - \Re_\theta \left( \sum_{d \mid y'; d < d_0} d^{-r} \chi(d) +  \sum_{d \mid y'; d \geq d_0+2; d \text{ odd}} d^{-r}\chi(d)\right) \\
        &= d_0^{-r} + \Re_\theta\left( \sum_{d \mid x'; d \geq d_0+2; d \text{ odd}} d^{-r}\chi(d) - \sum_{d \mid y'; d \geq d_0+2; d \text{ odd}} d^{-r}\chi(d)\right) \\
        &\geq d_0^{-r} - \sum_{d \geq d_0 + 2 ; d \text{ odd}} d^{-r},
    \end{align*}
    which is positive by Lemma \ref{tail}. Thus, we have $C_{r, \Q, \chi} \geq |A(\mathbb N)|.$ 

    It remains to give a lower bound for $|A(\N)|.$ Let $P$ be the set of odd primes $p \leq 2r-2$ that are coprime to $m$. For each subset $S \subseteq P,$ define $$n_S = \prod_{p \in S} p .$$ Because $n_S$ is odd, $2 \not\in B_1(n_S),$ so $B_2(n_S) = \{d \in \N : d \mid n_S, r-1 \leq d \leq 2r-2, \gcd(d, 2m) = 1\}.$ Each prime $p \in S$ is recorded either in $B_1(n_S)$ (if $p \leq r-1$) or in $B_2(n_S)$ (if $r-1\leq p \leq 2r-2$). Thus, distinct choices of subsets $S$ of $P$ give rise to distinct values of $A(n_S)$. Consequently, we have $$C_{r, \Q, \chi} \geq |A(\mathbb N)| \geq 2^{|P|} = 2^{\pi_m(2r-2) - \delta_m},$$ where $\delta_m$ equals 0 if $m$ is even and equals 1 if $m$ is odd.
\end{proof}

\section{A lower bound for $C_{r,K}$} \label{section-lower-2}

In this section, we prove Theorem \ref{lowerbound2}, which gives a lower bound for $C_{r,K}$ for finite Galois extensions $K$ of $\Q$ and generalizes the lower bound that \cite[Theorem 4.1]{AB19} gives when $K = \Q$. The proof's key mechanism is inspired by that of the previous section and of \cite[Theorem 4.1]{AB19}. We partition the set of integral ideals in \(\mathcal O_K\)
based on the multiplicities of their divisor ideals whose norms are bounded by a
function \(h_K(r)\) depending on \(K\) and \(r\). This partition captures explicit
gaps in the range of the associated divisor function; thus, the problem reduces to
lower-bounding \(h_K(r)\) and counting how many distinct partition classes are
created.
The new feature when handling a number field \(K\), rather than just \(\mathbb Q\),
is that distinct ideals can have the same norm. Then in order to bound
\(h_K(r)\), we control the number of ideals of each fixed norm using the
coefficients of the Dedekind zeta function. Finally, we use Chebotarev's density
theorem to estimate the number of rational primes that split completely in \(K\), giving a 
lower bound for the number of gaps created.

 We begin by defining some notation that we use throughout this section. We fix a degree-$s$ Galois extension $K/ \Q. $ 
For $n \geq 1,$ we define $a_K(n)$ to be the number of integral ideals $I$ in $K$ such that $N(I) = n.$ Then the Dedekind zeta function can be written as $$\zeta_K(r) = \sum_{n=1}^\infty \frac{a_K(n)}{n^r}.$$ In addition, we define $b_K(n)$ be the number of prime ideals $\mathfrak p$ in $K$ such that $N(\mathfrak p) = n.$ Lastly, for any ideal $I \subseteq \mathcal O_K,$ we define $m_I(n)$ to be the number of integral ideals $J$ in $K$ for which $J \mid I$ and $N(J) = n$. Then the divisor function can be written as $$\sigma_{r,K}(I) = \sum_{n=1}^\infty \frac{m_I(n)}{n^r}.$$

We first prove a general ``theoretical" lower bound on $C_{r,K}$ that generalizes the partition method used in the previous section. 

\begin{definition}
    Fix $r > 1.$ We define $h_K(r)$ to be the largest integer $h \geq 1$ such that for each integer $2 \leq d \leq h,$ the inequality $$d^{-r} > \sum_{n > d} \frac{a_K(n)}{n^r}$$ holds. If no such $h \geq 2$ exists, then define $h_K(r) = 1.$
\end{definition}

\begin{theorem} \label{partition}
    For any $r > 1,$ we have $$C_{r,K} \geq \prod_{n=2}^{h_K(r)} (b_K(n) + 1).$$
\end{theorem}

\begin{proof}
    Let $h = h_K(r).$ Define the map $A : I_K \ra \Z_{\geq 0}^{h-1}$ by $$A(I) = (m_I(2), m_I(3), \ldots, m_I(h)).$$ Note that this determines a partition of $I_K$ whose parts are $A^{-1}(S)$ for $S \in A(I_K).$ We wish to show that for any distinct $S_1, S_2 \in A(I_K),$ the sets $\sigma_{r,K}(A^{-1}(S_1))$ and $\sigma_{r,K}(A^{-1}(S_2))$ have disjoint closures. Write $S_1 = A(\mathfrak a)$ and $S_2 = A(\mathfrak b)$ for distinct $\mathfrak a, \mathfrak b \in I_K.$ Let $d_0$ be the smallest integer $2 \leq d_0 \leq h$ such that $m_\fa(d_0) \neq m_\bb(d_0),$ and we assume without loss of generality that $m_\fa(d_0) > m_\bb(d_0).$ 

    Now, let $\fa'$ and $\bb'$ be ideals in $K$ such that $A(\fa') = A(\fa)$ and $A(\bb') = A(\bb).$ From the definition of $d_0,$ we have $m_{\fa'}(d_0) - m_{\bb'}(d_0) \geq 1.$ Thus, we observe 
    \begin{align*}
        \sigma_{r,K}(\fa') - \sigma_{r,K}(\bb') &= \sum_{n=1}^\infty \frac{m_{\fa'}(n) - m_{\bb'}(n)}{n^r} \\
        &\geq d_0^{-r} - \sum_{n = d_0+1}^\infty \frac{m_{\bb'}(n)}{n^r} \\
        &\geq d_0^{-r} - \sum_{n = d_0+1}^\infty \frac{a_K(n)}{n^r}.
    \end{align*} Because $d_0 \leq h_K(r),$ the definition of $h_K(r)$ implies that the last expression above is positive and thus that $\sigma_{r,K}(\fa') > \sigma_{r,K}(\bb').$

    Thus, we have shown that the closures of $\sigma_{r,K}(A^{-1}(S_1))$ and $\sigma_{r,K}(A^{-1}(S_2))$ are disjoint. It follows that $C_{r,K} \geq |A(I_K)|,$ so it remains to determine a lower bound for $|A(I_K)|.$ We do so in a fashion analogous to that of the previous section. For each $2 \leq n \leq h,$ denote the prime ideals of norm $n$ by $$\mathfrak{p}_{n, 1}, \ldots, \mathfrak{p}_{n, b_K(n)}.$$ For each tuple $(i_n)_{2 \leq n \leq h},$ where $0 \leq i_n \leq b_K(n),$ define the ideal $$I_{(i_n)} = \prod_{n=2}^h \prod_{j=1}^{i_n} \mathfrak{p}_{n, j}.$$ We check that distinct tuples $(i_n)$ give distinct values of $A(I_{(i_n)}).$ Let $\textbf{i} =(i_n)$ and $\textbf{j} = (j_n)$ be distinct tuples, and let $n_0$ be the smallest integer for which $i_{n_0} \neq j_{n_0}.$ We show $m_{I_\textbf{i}}(n_0) \neq m_{I_\textbf{j}}(n_0).$ The prime ideals dividing $I_\textbf{i}$ and $I_\textbf{j}$ that are of norm less than $n_0$ are identical. Thus, the ideals dividing $I_\textbf{i}$ and $I_\textbf{j}$ that are not prime and are of norm equal to $n_0$ are also identical. Hence, $m_{I_\textbf{i}}(n_0) - m_{I_\textbf{j}}(n_0)$ equals the difference in the number of prime ideals dividing $I_\textbf{i}$ and $I_\textbf{j}$ that have norm $n_0,$ which is $i_{n_0} - j_{n_0} \neq 0.$

    Therefore, $$C_{r,K} \geq \#\{(i_n)_{2 \leq n \leq h} : 0 \leq i_n \leq b_K(n)\} = \prod_{n=2}^h (b_K(n) + 1).$$\end{proof}

Now, we aim to bound $a_K(n)$ in terms of $n$.
To do this, we first define the function $$d_s(n) = \#\{(n_1, \ldots, n_s) \in \mathbb N^s : n_1 n_2 \cdots n_s = n\}$$ for any positive integer $s. $ The function $d_s(n)$ is clearly multiplicative. For any prime $p$ and nonnegative integer $e$, we see that $$d_s(p^e) = \# \{(e_1, \ldots, e_s) \in \Z_{\geq 0} : e_1 + \dots + e_s = e\} = \binom{e + s-1}{s-1}.$$ Therefore, we have the general formula $$d_s(n) = \prod_{p^{e_p} \|n} \binom{e_p+s-1}{s-1}.$$ 

\begin{lemma}
    Let $K/\Q$ be a degree-$s$ Galois field extension. For any $n \geq 1,$ we have $a_K(n) \leq d_s(n).$
\end{lemma}

\begin{proof}
    The Dedekind zeta function can be expressed as $$\zeta_K(x) = \prod_{\mathfrak p \subseteq \mathcal O_K} (1 - N(\mathfrak p)^{-x})^{-1}.$$ Because $a_K$ and $d_s$ are multiplicative, we fix a rational prime $p$ and prove $a_K(p^a) \leq d_s(p^a)$ for any positive integer $a$. Let $f_p$ be the inertia degree of a prime ideal in $K$ lying above $p$, and let $r_p$ be the number of distinct prime ideals in $K$ lying above $p$. Then $$\zeta_K(x) = \prod_p  (1 - p^{-f_p x})^{-r_p}.$$ 
    Note that $a_K(p^a)$ is precisely the coefficient of $p^{-ax}$ in $\zeta_K(x)$ and is thus the coefficient of $p^{-ax}$ in the $p$-Euler factor $(1 - p^{-f_p x})^{-r_p}.$ Because $f_p \geq 1$ and $r_p \leq s,$ the coefficient of $p^{-ax}$ in the expansion of $(1 - p^{-f_p x})^{-r_p}$ is less than or equal the coefficient of $p^{-ax}$ in the expansion of $(1 - p^{-x})^{-s}.$ The latter coefficient is precisely $d_s(p^a),$ as desired.
\end{proof}

Next, we bound $d_s(n)$ in terms of $n$. 

\begin{lemma}
    Fix a positive integer $s$. For any $\varepsilon > 0,$ there exists a constant $\eta_{s, \varepsilon} > 0$ such that $d_s(n) \leq \eta_{s, \varepsilon} n^\varepsilon$ for any $n \geq 1.$
\end{lemma}

\begin{proof}
    Fix $\varepsilon > 0.$ Recall the closed formula $$d_s(n) = \prod_{p^{e_p} \|n} \binom{e_p+s-1}{s-1}.$$ Each binomial coefficient $\binom{e+s-1}{s-1}$ is bounded above by $e^{s-1}$ for large enough $e$. This polynomial growth is dominated eventually by exponential growth, so there exists $E$ such that if $e > E,$ then $$\binom{e+s-1}{s-1} \leq 2^{\frac{e \varepsilon}{2}}.$$ We then split the product in $d_s(n)$ using $E$; that is, we write $d_s(n) = d_{s,1}(n) d_{s,2}(n),$ where we define $$d_{s,1}(n) = \prod_{p^{e_p} \|n, 1 \leq e_p < E} \binom{e_p+s-1}{s-1} \text{ and }d_{s,2}(n) = \prod_{p^{e_p} \|n, e_p \geq E} \binom{e_p+s-1}{s-1}.$$ We can bound the second product as follows. $$d_{s,2}(n) \leq \prod_{p^{e_p} \| n} 2^\frac{e_p \varepsilon}{2} \leq \prod_{p^{e_p} \| n} p^\frac{e_p \varepsilon}{2} = n^\frac{\varepsilon}{2}$$ We proceed to bound the first product. Let $$M = \max_{1 \leq e < E} \binom{e+s-1}{s-1}.$$ Then $d_{s,1}(n) \leq M^{\omega(n)},$ where $\omega(n)$ is the number of distinct prime factors of $n$. Using the standard upper bound $\omega(n) = O(\frac{\log n}{\log \log n}),$ we find that 
    \begin{align*}
        d_{s,1}(n) & \leq \exp(\log M \cdot \omega(n)) \\
        &= \exp\left( O_{s, \varepsilon} \left( \frac{\log n}{\log \log n}\right)\right) \\
        &\leq \exp\left( \delta_{s, \varepsilon} \frac{\log n}{\log \log n}\right) \text{ for some constant $\delta_{s,\varepsilon}$ and large enough }n \\
        &\leq \exp\left( \frac{\varepsilon}{2} \log n \right)\text{ for large enough }n.
    \end{align*} Therefore, we have $d_s(n)  = d_{s,1}(n) d_{s,2}(n) \leq n^\varepsilon$ for large enough $n$. Consequently, accounting for the finitely many smaller values of $n$, there exists a constant $\eta_{s, \varepsilon} > 0$ such that $d_s(n) \leq \eta_{s, \varepsilon} n^\varepsilon$ for any $n \geq 1.$
\end{proof}

Thus, we have shown that for any finite Galois extension $K/\Q$ and $\varepsilon > 0,$ there exists a constant $\eta_{K, \varepsilon} := \eta_{s,\varepsilon}$ such that $a_K(n) \leq \eta_{K, \varepsilon} n^\varepsilon$ for any positive integer $n$. We use this bound to show that the function $h_{K, \varepsilon}(r)$ we define next is a lower bound for $h_K(r)$. Then we apply Theorem \ref{partition} and Chebotarev's density theorem  to prove Theorem \ref{lowerbound2}. 

\begin{definition}
    Let $K/\Q$ be a finite Galois extension, and let $\varepsilon > 0.$ For $r > \varepsilon + 1,$ define $$h_{K, \varepsilon}(r) = \left\lfloor \left( \frac{r-\varepsilon-1}{\eta_{K,\varepsilon}} \right)^{\frac{1}{1+\varepsilon}} \right\rfloor.$$
\end{definition}

\begin{proof}[Proof of Theorem \ref{lowerbound2}]
    We show $h_{K, \varepsilon}(r) \leq h_K(r).$ Let $2 \leq d \leq h_{K, \varepsilon}(r).$ Because $a_K(n) \leq \eta_{K, \varepsilon}n^\varepsilon$ for any positive integer $n$, we have \begin{align*}
        \sum_{n > d} \frac{a_K(n)}{n^r} &\leq \eta_{K, \varepsilon} \sum_{n > d} n^{-r + \varepsilon} \\
        &< \eta_{K, \varepsilon}\int_d^\infty x^{-r+\varepsilon} \; dx \\
        &= \frac{\eta_{K, \varepsilon}}{r-\varepsilon-1} d^{-r+\varepsilon+1},
    \end{align*} where the last equality holds because $r > \varepsilon + 1.$ The fact that $d \leq h_{K, \varepsilon}(r)$ implies $$\frac{\eta_{K, \varepsilon}}{r-\varepsilon-1} d^{-r + \varepsilon + 1} \leq d^{-r}.$$ The computation above thus gives $$\sum_{n > d} \frac{a_K(n)}{n^r} < d^{-r},$$ verifying $h_{K, \varepsilon}(r) \leq h_K(r).$

    Recall that by Theorem \ref{partition}, we have $$C_{r,K} \geq \prod_{n=2}^{h_{K,\varepsilon}(r)} (b_K(n) + 1).$$ We then restrict this product to the rational primes that split completely in $K$ in order to apply Chebotarev's density theorem. In particular, we have \begin{align*}
        C_{r,K} &\geq \prod_{p \leq h_{K, \varepsilon}(r) \text{ splits completely in }K} (b_K(p) + 1) \\
        &= (s+1)^{\#\{ p \leq h_{K, \varepsilon}(r) : p \text{ splits completely in }K\}}.
    \end{align*} Notice  that $h_{K, \varepsilon}(r) \asymp_{K,\varepsilon} r^{1/ (1 + \varepsilon)}$; that is, $h_{K,\varepsilon}(r)$ is bounded above and below by $(K,\varepsilon)$-dependent constant multiples of $r^{1/(1+\varepsilon)}$ for $r$ sufficiently large.
    Chebotarev's density theorem then gives us 
    \begin{align*}
        \#\{ p \leq h_{K, \varepsilon}(r) : p \text{ splits completely in }K\} &\sim \frac{1}{s} \frac{h_{K,\varepsilon}(r)}{\log h_{K, \varepsilon}(r)} \\
        &\asymp_{K,\varepsilon} \frac{r^{1/(1+ \varepsilon)}}{\log r}.
    \end{align*} Because $s$ is dependent on $K$, we conclude that there exist constants $\mu_{K,\varepsilon}$ and $r_{K,\varepsilon}$ for which $$\log C_{r,K} > \mu_{K,\varepsilon} \left( \frac{r^{1/(1+ \varepsilon)}}{\log r} \right)$$ holds for any $r > r_{K,\varepsilon}.$
\end{proof}

\section{Unboundedness of $C_{r,K}$ when $r$ and $[K:\Q]$ are fixed}\label{section-last}

The previous section showed that for a fixed number field $K$, the number of connected components $C_{r,K}$ tends to $\infty$ as $r$ tends to $\infty.$ A natural next question is whether $C_{r,K}$ exhibits the same behavior when $r$ is fixed and $K$ is varied. In this section, we prove Theorem \ref{fix-s}, which answers this question in the positive. Indeed, we show the stronger result that $C_{r,K}$ can be arbitrarily large when $r$ and $[K:\Q]$ are fixed and $K$ is varied.
To prove Theorem \ref{fix-s}, we define a notion of \textit{mighty norms} that generalizes terminology defined in \cite{Z18}, after which we outline the main ideas of the proof.  

\begin{definition}
    Let $d > 1$ be the norm of some prime ideal in $K$. We say $d$ is an \textit{$r$-mighty norm} of $K$ if \begin{equation} \label{mighty}
        1 + d^{-r} > \prod_{N(\mathfrak p) > d} (1 - N(\mathfrak p)^{-r})^{-1}.
    \end{equation}
\end{definition}

Like Zubrilina's definition, the existence of $r$-mighty norms correspond to the existence of gaps in $\sigma_{r,K}(I_K),$ which we make precise below. 

\begin{lemma}\label{lemma-gaps}
    Let $d_1 < d_2 < \dots < d_M$ be $r$-mighty norms of $K$. Then $C_{r, K} \geq M+1.$
\end{lemma}

\begin{proof}
    Fix $d := d_i$ for some $1 \leq i \leq M.$ We first show the existence of a gap in $\overline{\sigma_{r,K}(I_K)}$ associated to $d$. Let $S$ be the set of ideals $I$ in $K$ such that some prime ideal factor $\mathfrak q$ of $I$ satisfies $N(\mathfrak q) \leq d.$ For any $I \in S,$ we have $\sigma_{r,K}(I) \geq 1+d^{-r}.$ Meanwhile, for any $I \in I_K \setminus S,$ we have $\sigma_{r,K}(I) \leq \prod_{N(\mathfrak p) > d} (1 - N(\mathfrak p)^{-r})^{-1}.$ Therefore, $\overline{\sigma_{r,K}(S)} \cap \overline{\sigma_{r,K}(I_K \setminus S)} = \emptyset.$ 

    Now, we verify that for any $1 \leq i < j \leq M,$ the gap intervals $$\left( \prod_{N(\mathfrak p) > d_j} (1 - N(\mathfrak p)^{-r})^{-1}, 1 + d_j^{-r} \right) \text{\; and \;} \left( \prod_{N(\mathfrak p) > d_i} (1 - N(\mathfrak p)^{-r})^{-1}, 1 + d_i^{-r} \right)$$ are disjoint. We see easily that $$\prod_{N(\mathfrak p) > d_i} (1 - N(\mathfrak p)^{-r})^{-1} \geq (1 - d_j^{-r})^{-1} > 1 + d_j^{-r},$$ so the intervals are disjoint. Thus, there are at least $M$ many gap intervals in $\overline{\sigma_{r,K}(I_K)},$ so $C_{r,K} \geq M+1.$
\end{proof}

For an arbitrary positive integer $M$, we aim to construct a sequence of rational primes $p_1 < p_2 < \cdots < p_M$ and a degree-$s$ number field $K$ in which $p_1,\ldots,p_M$ are $r$-mighty norms. By Lemma \ref{lemma-gaps}, this would immediately give Theorem \ref{fix-s}.
The idea is to exploit the effect of a prime ideal's splitting behavior on its norm. Notice that rational primes $p$ that split completely in a number field $K/\Q$ have ``smaller" norms $p$, whereas rational primes that are inert in $K$ have ``larger" norms $p^s$. Hence, if the primes $p_1 < p_2 < \dots < p_M$ are spaced widely enough apart and if they split completely while nearby primes remain inert in $K$, then $p_1,p_2,\ldots,p_M$ are $r$-mighty norms of $K$. The first lemma we introduce achieves the first goal of constructing such a sequence of rational primes, and the second lemma we introduce achieves the second goal of showing the existence of a number field $K$ with the desired splitting behaviors. 

\begin{lemma} \label{technical}
    Fix a real number $r > 1,$ an integer $s \geq 2,$ and an integer $M \geq 1.$ There exists a sequence of rational primes $p_1 < p_2 < \dots < p_M$ satisfying the following conditions. 
    \begin{enumerate}
        \item \begin{enumerate}
            \item $p_1 > \max(2,s)^s.$
            \item $\log(1-x)^{-1} \leq 2x$ for $x = p_1^{-r}.$
            \item $\log(1+x) \geq \frac{x}{2}$ for $x = p_1^{-r}.$
        \end{enumerate} \noindent (Note that (b) and (c) thus also hold for any $0 \leq x \leq p_1^{-r}.$)
        \item For each $2 \leq i \leq M,$ we have $p_i^{1/s} > p_{i-1}.$ In particular, defining for each $1 \leq i \leq M$ the set $S_i = \{q \text{ prime} : p_i^{1/s} < q \leq p_i\},$ the sets $S_1,\ldots,S_M$ are pairwise disjoint.
        \item For any $1 \leq j < i \leq M,$ we have $$s \cdot \sum_{q \in S_i} q^{-r} < \frac{1}{12M} p_j^{-r}.$$
        \item For any $1 \leq i \leq M,$ we have $$\sum_{\text{prime }q > p_i} q^{-sr} < \frac{1}{12}p_i^{-r}.$$
    \end{enumerate}
\end{lemma}

Before we give a technical proof of this lemma, we briefly explain the later necessity of each condition. Recall that our goal is to show $p_1, p_2, \ldots, p_M$ are $r$-mighty norms for some degree-$s$ number field $K$.
Condition (1) allows us to compare the right-hand side of Equation (\ref{mighty}) for $d = p_i$ with simpler expressions in $p_i,$ such as $\sum_{N(\mathfrak p) > p_i} N(\mathfrak p)^{-r}.$ Then conditions (2), (3), and (4) allow us to give an upper bound on $\sum_{N(\mathfrak p) > p_i} N(\mathfrak p)^{-r}$ that we then compare to the left-hand side of Equation (\ref{mighty}) for $d = p_i.$ 

\begin{proof}[Proof of Lemma \ref{technical}]
    We construct the desired sequence of primes recursively. We first construct the desired $p_1.$ Condition (1) can be imposed by taking $p_1$ sufficiently large, because $$\lim_{x \ra 0} \frac{\log(1-x)^{-1}}{x} = \lim_{x \ra 0} \frac{\log(1+x)}{x} = 1.$$
    
    Now, suppose we have chosen $p_1, \ldots, p_{i-1}$ to satisfy the above conditions; we aim to choose $p_i$ such that $p_1,\ldots,p_{i}$ satisfy the four conditions. Condition (2) clearly just requires $p_i$ to be sufficiently large. Next, we consider condition (3). Because $r > 1,$ we may write $$\sum_{q \in S_i} q^{-r} \leq \sum_{q  > p_i^{1/s}} q^{-r} < \int_{p_i^{1/s}}^\infty x^{-r} \;dx = \frac{(p_i^{1/s})^{-r+1}}{r-1} \ll_r p_i^{\frac{1-r}{s}}.$$ Therefore, we have $$s \cdot \sum_{q \in S_i} q^{-r}  \ll_{r, s} p_i^{\frac{1-r}{s}}.$$ The right-hand side tends to 0 as $p_i$ tends to $\infty,$ so we can choose $p_i$ sufficiently large to bound it from above by $\frac{1}{12M} \min_{1 \leq j < i} p_j^{-r}.$ Lastly, we consider condition (4). Similarly, because $s \geq 2$ and $r > 1$ and thus $sr-1 > 1,$ we see that $$\sum_{\text{prime }q > p_i} q^{-sr} \ll_r p_i^{-sr+1}.$$ Because $sr-1 > r,$ we can choose $p_i$ sufficiently large to bound the left-hand side from above by $\frac{1}{12}p_i^{-r}.$
\end{proof}

\begin{lemma} \label{ntlemma}
    Let $s \geq 2$ be an integer, and let $S$ and $T$ be disjoint finite sets of rational primes. Furthermore, assume each prime in $S$ is greater than $s$. Then there exists a degree-$s$ extension $K/\Q$ such that each prime in $S$ splits completely in $K$ and each prime in $T$ remains inert in $K$. 
\end{lemma}
\begin{proof}
    We aim to construct a monic irreducible polynomial  $f(x) \in \Z[x]$ such that, modulo each prime in $S \cup T,$ its factorization type corresponds to the desired splitting behavior. We then can apply the Dedekind--Kummer theorem to show that $\Q(\alpha)$, where $\alpha$ is a root of $f(x),$ is a desired choice of $K$.

    For each $p \in S,$ we choose a monic degree-$s$ polynomial $f_p(x) = x^s + c_{p,s-1}x^{s-1} + \cdots + c_{p,0} \in \mathbb F_p[x]$ that splits into $s$ distinct linear factors over $\mathbb F_p.$ This is feasible because $s < p.$ Furthermore, for each $p \in T,$ we choose a monic irreducible polynomial $f_p(x)  = x^s + c_{p,s-1} x^{s-1} + \cdots + c_{p,0} \in \mathbb F_p[x]$ of degree $s$. Finally, fix some prime $q \not\in S \cup T,$ and let $f_q(x) = x^s + q.$ We show there exists a monic degree-$s$ polynomial $f(x) \in \Z[x]$ such that 
    $$f(x) \equiv \begin{cases}
        f_p(x) \pmod{p}& \text{ for } p \in S \cup T \\
        f_q(x) \pmod{q^2}
    \end{cases}$$
    For each $0 \leq i \leq s-1,$ we apply the Chinese remainder theorem to obtain $n_i \in \Z$ such that $n_i \equiv c_{p,i} \pmod{p}$ for each $p \in S \cup T$ and $n_i \equiv q \pmod{q^2}$ if $i = 0$ and $n_i \equiv 0 \pmod{q^2}$ else. Then $f(x) = x^s + n_{s-1}x^{s-1} + \cdots + n_0 \in \Z[x]$ satisfies the desired congruences.  
    
    Note that $f$ is an Eisenstein polynomial because $f(x) \equiv x^s + q \pmod{q^2},$ so it is irreducible. Let $\alpha$ be a root of $f,$ and let $K  = \Q(\alpha).$ To apply the Dedekind--Kummer theorem to $K$ for each $p \in S \cup T$, we must check that $p \nmid [\mathcal O_K : \Z[\alpha]].$ Fix $p \in S \cup T$. We have the identity $$\text{disc}(f) = [\mathcal O_K : \Z[\alpha]]^2 \cdot \text{disc}(K).$$ Now, notice that $f(x) \pmod{p}$ is squarefree; when $p \in S,$ this is by construction, and when $p \in T,$ we note that any irreducible polynomial over a finite field is separable and thus squarefree. Therefore, $p \nmid \text{disc}(f)$. Using the fact that $\text{disc}(K) \in \Z$ because $\alpha$ is an algebraic integer, we obtain $p \nmid [\mathcal O_K : \Z[\alpha]].$ Thus, the Dedekind--Kummer theorem applies at each $p \in S \cup T.$ If $p \in S,$ then $f(x) \equiv f_p(x) \pmod{p}$ splits into $s$ distinct linear factors, so $p$ splits completely in $K$; if $p \in T,$ then $f(x) \equiv f_p(x) \pmod{p}$ is irreducible, so $p$ remains inert in $K$. 
\end{proof}

We are ready to prove Theorem \ref{fix-s}.

\begin{proof}[Proof of Theorem \ref{fix-s}]
    Let $M$ be an arbitrary positive integer. By Lemma \ref{technical}, we have a sequence of rational primes $p_1 < p_2 < \dots < p_M$ satisfying the four conditions specified in the lemma's statement. Now, note that we can choose a real number $X > p_M$ such that $$s \cdot \sum_{q > X} q^{-r} < \frac{1}{12} p_M^{-r}.$$ The left-hand side is bounded above by some $r$-dependent constant multiple of $X^{-r+1}$ and thus tends to 0 as $X$ tends to $\infty,$ so making $X$ large enough ensures the validity of the above inequality. 
    Now, let $$S = \bigcup_{i=1}^M S_i,$$ where $S_i$ are defined in the statement of Lemma \ref{technical}, and let $T$ be the set of rational primes $q \not\in S$ such that $q \leq X.$ Applying Lemma \ref{ntlemma} to these choices of $S$ and $T$, we have a degree-$s$ extension $K/\Q$ such that $p_1, \ldots, p_M$ split completely in $K$ and any rational prime $q \not\in S$ for which $q \leq X$ is inert.  

    Fix $1 \leq i \leq M.$ Because $p_i$ splits completely in $K$, it is a norm of a prime ideal in $K$. We prove that $p_i$ is an $r$-mighty norm of $K$; this would imply $C_{r,K} \geq M+1,$ proving the theorem. Recall that this requires showing $$\prod_{N(\mathfrak p) > p_i} (1 - N(\mathfrak p)^{-r})^{-1} < 1 + p_i^{-r}.$$ By condition (1) in Lemma \ref{technical}, we have $\log(1-x)^{-1} \leq 2x$ for any $x \leq p_1^{-r}.$ Therefore, it suffices to show \begin{equation} \label{eq-suffices}
        2 \sum_{N(\mathfrak p) > p_i} N(\mathfrak p)^{-r} < \log(1 + p_i^{-r}).
    \end{equation} We proceed to bound the sum on the left-hand side by considering the three following types of rational primes $q$ that $\mathfrak p$ can lie over. 
    \begin{itemize}
        \item $q \leq X$ is a rational prime that splits completely in $K$. The norm of any prime ideal lying above $q$ is $q$, so by condition (2) of Lemma \ref{technical}, if $q > p_i,$ then $q \in S_j$ for some $j > i$. Therefore, the total contribution of prime ideals $\mathfrak p$ lying above such $q$ satisfying $N(\mathfrak p) > p_i$ to the summation is bounded from above by $$s \cdot \sum_{j > i} \sum_{q \in S_j} q^{-r} < \sum_{j > i} \frac{1}{12M} p_i^{-r} < \frac{1}{12}p_i^{-r},$$ where the first inequality is condition (3) of Lemma \ref{technical}.
        \item $q \leq X$ is a rational prime that remains inert in $K$. If $q \leq p_i^{1/s},$ then $N(q\mathcal O_K) = q^s \leq p_i,$ so $q \mathcal O_K$ does not contribute to the summation. Because $q \leq X$ is inert in $K$, it does not lie in $S_i,$ so $q \not\in (p_i^{1/s}, p_i].$ Lastly, if $q > p_i$, then the sole prime ideal $q \mathcal O_K$ lying above $q$ contributes $q^{-sr}$ to the summation. The total contribution of these inert $q \leq X$ is $$\sum_{\text{prime }q > p_i} q^{-sr} < \frac{1}{12} p_i^{-r},$$ where the inequality is condition (4) of Lemma \ref{technical}.
        \item $q > X$ is a rational prime. Then $q > p_i,$ so any prime ideal $\mathfrak p$ lying above $q$ satisfies $N(\mathfrak p) > p_i.$ There are at most $s$ prime ideals in $K$ lying above $q$, and each one has norm at least $q$. Thus, the contribution of such ideals to the summation is bounded above by $$s \cdot \sum_{q > X} q^{-r} < \frac{1}{12} p_M^{-r} \leq \frac{1}{12} p_i^{-r},$$ where the middle inequality follows from the definition of $X$.  
    \end{itemize}
    Summing the contributions from these three cases, we conclude that $$\sum_{N(\mathfrak p) > p_i} N(\mathfrak p)^{-r} < \frac{1}{4} p_i^{-r}.$$ Then Equation (\ref{eq-suffices}) holds if $$2 \left(\frac{1}{4} p_i^{-r}\right) = \frac{1}{2} p_i^{-r} < \log(1 + p_i^{-r}).$$ This holds true by condition (1) in Lemma \ref{technical}, so we have shown $p_i$ is an $r$-mighty norm of $K$. 
\end{proof}

\section{Acknowledgments}

This research was conducted at the Duluth REU at the University of Minnesota Duluth, which is supported by National Science Foundation grant No.~DMS-2409861, Jane
Street Capital, and donations from Ray Sidney and Eric Wepsic. I would like
to thank Colin Defant, Noah Kravitz, Mitchell Lee, Maya Sankar, and Carl Schildkraut for providing guidance during the research process, as well as Derek Liu for helpful discussions. I would also like to thank Jonas Iskander, Noah Kravitz, Andrew Kwon, Mitchell Lee, and Yelena Mandelshtam for their detailed feedback throughout the editing process. Finally, I thank Colin
Defant and Joe Gallian for providing this wonderful opportunity to participate in the Duluth REU.

\end{document}